\documentclass[12pt]{amsart}
\usepackage{amssymb, amsmath, amsthm} 
\usepackage{url}      
\usepackage{mathtools} 
\usepackage{tikz-cd}
\usepackage[T1]{fontenc}
\usepackage[utf8]{inputenc}
\usepackage{xcolor}
\usepackage{mathrsfs}
\usepackage{stmaryrd}
\usepackage{eucal}
\usepackage[all]{xy}
\usepackage[margin=1in]{geometry} 
\usepackage{hyperref}
\usepackage[dvipsnames]{xcolor}
\usepackage[shortlabels]{enumitem}
\hypersetup{bookmarksdepth=2}
\hypersetup{colorlinks=true}
\hypersetup{linkcolor=blue}
\hypersetup{citecolor=blue}
\hypersetup{urlcolor=blue}
\allowdisplaybreaks

\theoremstyle{plain}
\newtheorem{theorem}{Theorem}[section] 
\newtheorem{proposition}[theorem]{Proposition} 
\newtheorem{lemma}[theorem]{Lemma}
\newtheorem{corollary}[theorem]{Corollary}

\theoremstyle{definition}
\newtheorem{definition}[theorem]{Definition}

\theoremstyle{remark}
\newtheorem{remark}[theorem]{Remark}
\newtheorem*{acknowledgement}{Acknowledgements}

\newcommand{\g}{\gamma}

\newcommand{\f}{\varphi}

\renewcommand{\k}{\kappa}

\newcommand{\s}{\sigma}

\renewcommand{\t}{\tau}

\newcommand{\G}{\Gamma}

\newcommand{\NN}{\mathbb{N}}

\newcommand{\QQ}{\mathbb{Q}}

\newcommand{\ZZ}{\mathbb{Z}}

\newcommand{\Ann}{\operatorname{Ann}}

\newcommand{\End}{\operatorname{End}}

\newcommand{\grand}{{\textup{grand}}} 
\newcommand{\grano}{\mathcal{O}^{\grand}}

\newcommand{\Hom}{\operatorname{Hom}}
\newcommand{\id}{\operatorname{id}}

\newcommand{\lcm}{\operatorname{lcm}}

\newcommand{\Mat}{\operatorname{Mat}}

\newcommand{\Orbit}{\mathcal{O}}

\newcommand{\Per}{\operatorname{Per}}

\newcommand{\Qbar}{{\bar{\QQ}}}

\renewcommand{\setminus}{\smallsetminus}

\newcommand{\Span}{\operatorname{Span}}

\title{On Dense Orbit Transversality for Endomorphisms of Abelian Varieties}

\author{Kaiwen Lu}
\address{Department of Mathematics, Brown University, Providence, RI, USA}
\email{kaiwen\_lu@brown.edu}

\begin{document}

\begin{abstract}
Let $X/K$ be a smooth projective variety defined over a number field and $f:X\to X$ be a morphism defined over $K$.
Assuming there exists a point in $X(K)$ whose $f$-orbit is Zariski dense in $X$ and up to replacing $K$ by a finite extension,
Pasten and Silverman studied the distribution of grand $(f,K)$-orbits and proved that many sets of representatives of grand $(f,K)$-orbits on various classes of varieties are Zariski dense.
In particular, they showed that if $X$ is a geometrically simple abelian variety, then all such sets of representatives are Zariski dense.
We demonstrate the existence of a dense set of representatives for maps on all abelian varieties.

\end{abstract}

\maketitle

\section{Introduction}

Let $X$ be a variety over a field $K$ and let $f:X\to X$ be a self-map on $X$ defined over $K$.
A central object of study in dynamics is the set of grand $(f,K)$-orbits of points in $X(K)$, as they encode the information of how points move under $f$.
\begin{definition}
    The \emph{grand $(f,K)$-orbit} of a point $P\in X(K)$ is the set of points whose orbits eventually merge with the orbit of $P$:
    \[\grano_{f,K}(P):=\{Q\in X(K)\mid  f^n(Q)=f^m(P) \text{ for some $n,m \ge 0$}\}.\]
\end{definition}
A natural question to ask is then: how are the grand $(f,K)$-orbits distributed in $X$?
We know that the set of grand $(f,K)$-orbits in $X(K)$ partitions $X(K)$, so we need to pin down a notion of distribution before we proceed.
To precisely discuss this problem, Pasten and Silverman \cite{pastensilverman} introduced orbit propagation principles, which measure the distribution of grand orbits by the Zariski density of some sets of representatives.
Using terminologies introduced by Silverman, one set of representatives considered is the following:
\begin{definition}
An \emph{$(f,K)$-transversal of $X$} is a subset $S\subset X(K)$ such that every $(f,K)$-grand orbit of $X(K)$ contains exactly one point in $S$.
\end{definition}
With the above definition, orbit propagation principles (C2$\exists$) and (C2$\forall$) described in \cite{pastensilverman} can be rephrased as follows:
\begin{definition}[Weak DOT]
We say that $(f,K)$ is \emph{weakly dense orbit transversal} (W-DOT) if for some finite extension $L/K$, \textbf{there exists} a Zariski dense $(f,L)$-transversal of $X$.
\end{definition}

\begin{definition}[Strong DOT]
We say that $(f,K)$ is \emph{strongly dense orbit transversal} (S-DOT) if for some finite extension $L/K$, \textbf{every} $(f,L)$-transversal of $X$ is Zariski dense.
\end{definition}

Pasten and Silverman proved orbit propagation principles of various strength for many classes of self-maps on families of projective varieties.
In the case of abelian varieties, they proved strongly dense orbit transversality on simple abelian varieties:
\begin{theorem}[\cite{pastensilverman} Theorem 6.1, rephrased]
Let $X/\Qbar$ be a geometrically simple abelian variety, and let $f:X\to X$ be an endomorphism of $X$ (as an abstract variety) such that there is a point in $X(\Qbar)$ with $f$-orbit that is Zariski dense in $X$.
Then there is a number field~$K/\QQ$ such that~$X$ and~$f$ are defined over~$K$ and $(f,K)$ is strongly dense orbit transversal.
\end{theorem}
The main result of this paper is weakly dense orbit transversality in a more general setting.
Namely, we demonstrate the existence of a dense $(f,K)$-transversal for all abelian varieties, not just geometrically simple ones.
\begin{theorem}\label{main}
Let $X/\Qbar$ be an abelian variety, and let $f:X\to X$ be an endomorphism of~$X$ (as an abstract variety) such that there is a point in $X(\Qbar)$ with $f$-orbit that is Zariski dense in $X$.
Then there is a number field~$K/\QQ$ such that~$X$ and~$f$ are defined over~$K$, and $(f,K)$ is weakly dense orbit transversal.
\end{theorem}

As another motivation for the defintion of $(f,K)$-transversals, Pasten and Silverman proved that over a number field $K$ and up to replacing $K$ with a finite extension, if $X(K)$ is Zariski dense in $X$, then the existence of a dense $(f,K)$-transversal is equivalent to $X(K)\setminus(\Gamma_1\cup\cdots \cup \Gamma_r)$ being Zariski dense for every finite collection of grand $(f,K)$-orbits $\Gamma_1,\ldots,\Gamma_r$.
This gives more evidence for why the Zariski density of an $(f,K)$-transversal is a reasonable notion for measuring the distribution of grand $(f,K)$-orbits.

One thing to note is that their proof for the aforementioned equivalence crucially relies on the fact that $K$ is countable. 
In fact, if $K$ is uncountable, then we will prove that strongly dense orbit transversality always holds for dominant morphisms of irreducible varieties; see Proposition \ref{uncountablefield}.

Henceforth we assume that $K$ is a number field, and we may abuse notation when replacing $K$ with a finite extension and still call the extension $K$.

\begin{acknowledgement}
    We would like to thank Joseph Silverman for many helpful discussions and comments and Dan Abramovich for the second example in Remark \ref{genus2}.
\end{acknowledgement}

\section{Reduction to simpler cases}
In this section, we prove several invariance results for the dense orbit transversal properties.
Then we construct lifts of self-maps on abelian varieties through isogenies, which combined with the invariance results allow us to reduce the proof of the main theorem for general abelian varieties to some abelian varieties with a simpler structure.

When we can lift $f$ through an isogeny to $f'$, it's natural to think that the dynamics of $f'$ is  finitely many copies of the dynamics of the dynamics of $f$.
This intuition is not far from being correct, as the lift of a grand $(f,K)$-orbit is a finite union of grand $(f',K)$-orbits of the preimages, which we prove in Lemma \ref{unionorbits}. 
This gives us a handle on analyzing the dynamics of $f$ through its lifts and descents.

Let $X,X'$ be abelian varieties over $K$, $\f:X'\to X$ an isogeny, and $f,f'$ self-maps on $X$ and $X'$ respectively such that the following diagram commutes:
\begin{equation}\label{isogliftdiag}
        \begin{tikzcd}
        X' \arrow[r, "f'"] \arrow[d, "\varphi"] & X' \arrow[d, "\varphi"] \\
        X \arrow[r, "f"]                                   & X                                 
\end{tikzcd}
\end{equation}
    By the rigidity theorem~\cite[Section~4, Corollary~1]{mumfordabvar}, we can decompose each of $f$ and $f'$ into a homomorphism followed by a translation:
    \[
        f(x)=\t(x)+c,\quad f'(x)=\t'(x)+c',
    \]
    where $\f(c')=c$.
Note that the commutativity of diagram (\ref{isogliftdiag}) implies that $\t'(\ker(\f))\subseteq \ker(\f)$, because for any $\k\in \ker(\f)$, we have 
    \[\f(\t'(\k))=\f(f'(\k)-c')=f(\f(\k))-\f(c')=f(0)-\f(c')=c-c=0.\]
    Since $\ker(\f)$ is finite,  for $N$ sufficiently large, the iterated image $\t'^N(\ker(\f))=\t'^{N+1}(\ker(\f))$ stabilizes,
    and $\t'$ acts on \[\ker(\f)_{\Per}:=\t'^N(\ker(\f))\]
    periodically.
\begin{lemma}\label{unionorbits}

    With the above notation,
    let $P\in X(K)$ be a point and $P'\in \f^{-1}(P)$, then the preimage of $\grano_{f,K}(P)$ under $\f$ is a union of the grand $(f',K)$-orbits of the preimages of $P$.
    In other words:
    \[\f^{-1}(\grano_{f,K}(P))=\bigcup_{\k\in \ker(\f)}\grano_{f',K}(P'+\k).\]

    If $P'$ is not $f'$-preperiodic, then
    the preimage of $\grano_{f,K}(P)$ under $\f$ is the disjoint union of exactly $\#\ker(\f)_{\Per}$ grand $(f',K)$-orbits indexed by points in $\ker(\f)_{\Per}$:
    \[\f^{-1}(\grano_{f,K}(P))=\bigsqcup_{\k\in \ker(\f)_{\Per}}\grano_{f',K}(P'+\k).\]
\end{lemma}
\begin{proof}

    We first show the inclusion $\bigcup_{\k\in \ker(\f)}\grano_{f',K}(P'+\k)\subseteq \f^{-1}(\grano_{f,K}(P))$.
    
    Let $y\in \bigcup_{\k\in \ker(\f)}\grano_{f',K}(P'+\k)$, then there exists $n,m\in \ZZ$ such that \[f'^n(y)=f'^m(P'+\k).\]
    By commutativity of the diagram, we have 
    \[f^n(\f(y))=f^m(\f(P'+\k))=f^m(P).\]
    Therefore $\f(y)\in \grano_{f,K}(P)$, and $y\in \f^{-1}(\grano_{f,K}(P))$ as desired.

    We next show the inclusion $\f^{-1}(\grano_{f,K}(P))\subseteq\bigcup_{\k\in \ker(\f)}\grano_{f',K}(P'+\k)$.
    
    Let $y\in \f^{-1}(\grano_{f,K}(P))$, then by definition of the grand orbit, there exists $n,m\in \ZZ$ such that 
    \[f^n(\f(y))=f^m(P)=f^m(\f(P')).\]
    By commutativity of the diagram, we have 
    \[\f(f'^n(y))=\f(f'^m(P')),\]
    which implies that, for some $\k'\in \ker(\f)$, we have
    \begin{equation}\label{lifteq}
        f'^n(y)=f'^m(P')+\k'.
    \end{equation}
    Since $f'$ is a regular map, it decomposes as the composition of a homomorphism $\t'$ followed by a translation by $c'$ by the rigidity theorem:
    \[f'(x)=\t'(x)+c'.\]
    We claim that for all $j\ge 0$, the following equation holds:
    \begin{equation}\label{compute}
        f'^j(x) = \t'^j(x) + \sum_{i=0}^{j-1} \t'^i(c').
    \end{equation}
    We will proceed by induction on $j$.
    The $j=0$ case is $x=x$.
    Now assume equation (\ref{compute}) holds for $j-1$, then
    \begin{align*}
        f'^j(x) =& f'(f'^{j-1}(x))\\
        =& f'(\t'^{j-1}(x) + \sum_{i=0}^{j-2} \t'^i(c'))\\
        =& \t'(\t'^{j-1}(x) + \sum_{i=0}^{j-2} \t'^i(c')) +c'\\
        =& \t'^{j}(x) + \left(\sum_{i=1}^{j-1} \t'^i(c')\right) +c'\\
        =& \t'^j(x) + \sum_{i=0}^{j-1} \t'^i(c').
    \end{align*}
    Hence for any two points $x,a\in X'$, we find
    \begin{equation}\label{fnxa}
        f'^n(x+a)=f'^n(x)+\t'^n(a)
    \end{equation} by the following calculation using equation (\ref{compute}):
    \begin{align*}
        f'^n(x+a)-f'^n(x)=& \t'^n(x+a)+\sum_{i=0}^{n-1} \t'^i(c)-(\t'^n(x)+\sum_{i=0}^{n-1} \t'^i(c))\\
        =&\t'^n((x+a)-(x))\\
        =&\t'^n(a).
    \end{align*}
    Previously we observed that $\t'$ acts periodically on the subgroup $\ker(\f)_{\Per}$.
    In particular, the map $\t'$ is an automorphism on $\ker(\f)_{\Per}$.
    Now we apply $f'^N$ to both sides of equation (\ref{lifteq}) and get 
    \[f'^{n+N}(y)=f'^{m+N}(P')+\t'^N(\k').\]
    Note that $\t'^N(\k')\in \ker(\f)_{\Per}$, so there exists $\k\in \ker(\f)_{\Per}$ such that $\t'^{m+N}(\k)=\t'^N(\k')$, and we find 
    \begin{align*}
        f'^{n+N}(y)=&f'^{m+N}(P')+\t'^N(\k')\\
        =&f'^{m+N}(P')+\t'^{m+N}(\k)\\
        =&f'^{m+N}(P'+\k).
    \end{align*}
    Therefore $y\in \bigcup_{\k\in \ker(\f)}\grano_{f',K}(P'+\k)$, proving the desired containment.

    As $\k\in \ker(\f)_{\Per}$, the argument above in fact shows that \[\f^{-1}(\grano_{f,K}(P))=\bigcup_{\k\in \ker(\f)_{\Per}}\grano_{f',K}(P'+\k),\]
    so we are left to show that if $P'$ not $f'$-preperiodic, then $\grano_{f',K}(P'+\k)\cap \grano_{f',K}(P'+\k')=\varnothing$ for $\k\ne \k'$ in $\ker(\f)_{\Per}$.

    Let $\k\ne \k'$ in $\ker(\f)_{\Per}$. 
    Suppose that $f'^n(P'+\k')=f'^n(P'+\k)$ for some $n$, then 
    \[f'^n(P')+\t'^n(\k')=f'^n(P')+\t'^n(\k),\]
    which implies that $\t'^n(\k')=\t'^n(\k)$ and contradicts the assumption that $\k\ne \k'$, because $\t'$ is an automorphism on $\ker(\f)_{\Per}$.
    
    Now suppose there exists $m>n$ such that $f'^m(P'+\k')=f'^n(P'+\k)$, then \[f'^m(P')+\t'^m(\k')=f'^n(P')+\t'^n(\k).\]
    Write $m-n=a$ and $\t'^m(\k')-\t'^n(\k)=\bar{\k}\in \ker(\f)_{\Per}$, then we find 
    $f'^n(P')=f'^m(P')+\bar{\k}$, which is the base case $k=1$ for the identity 
    \[f'^{n+ka}(P')=f'^{n}(P')-\sum_{i=0}^{k-1} \t'^{ia}(\bar{\k}).\]
    For the inductive step, assume the identity is true for $k-1$, we find that
    \begin{align*}
        f'^{n+ka}(P')&=f'^{a}(f'^{n+(k-1)a}(P'))\\
        &=f'^{a}\left(f'^n(P')-\left(\sum_{i=0}^{k-2} \t'^{ia}(\bar{\k})\right)\right)\\
        &=f'^{n+a}(P')-\t'^a\left(\sum_{i=0}^{k-2} \t'^{ia}(\bar{\k})\right)\\
        &=f'^{n}(P')-\bar{\k}-\left(\sum_{i=1}^{k-1} \t'^{ia}(\bar{\k})\right)\\
        &=f'^{n}(P')-\sum_{i=0}^{k-1} \t'^{ia}(\bar{\k}).
    \end{align*}
    Since $\ker(\f)_{\Per}$ is a finite group, we get that $\sum_{i=0}^{k-1} \t'^{ia}(\bar{\k})$ takes on finitely many values as we vary $k$, so $P'$ is preperiodic.
\end{proof}
We denote the $f$-orbit of a point $P\in X(K)$ as
\[\Orbit_f(P):=\{f^n(P)\mid n\ge 0\}.\]

The existence of a point with dense orbit, weakly dense orbit transversality, and strongly dense orbit transversality are all invariant under isogenies and taking iterates of the self-map $f$.
We will describe the precise sense of invariance in Lemmas \ref{isoginv} and \ref{iterinv}.
\begin{lemma}\label{isoginv}
Let $K/\mathbb{Q}$ be a number field, $X/K$ and $X'/K$ abelian varieties, and $\varphi:X'\to X$ an isogeny.
Let $f:X\to X$ and $f':X'\to X'$ be regular self-maps such that the following diagram commutes:
    \begin{equation}\label{isoginvdiag}
        \begin{tikzcd}
        X' \arrow[r, "f'"] \arrow[d, "\varphi"] & X' \arrow[d, "\varphi"] \\
        X \arrow[r, "f"]                                   & X                                 
\end{tikzcd}
    \end{equation}
    Then the following hold:
    \begin{enumerate}[label=\emph{(\alph*)}]
        \item There is a point $P\in X$ with $\overline{\mathcal{O}_f(P)}=X$ if and only if there is a point $P'\in X'$ with $\overline{\mathcal{O}_{f'}(P')}=X'$ .
        \item There exists a dense $(f,K)$-transversal in $X(K)$ if and only if there exists a dense $(f',K)$-transversal in $X'(K)$.
        \item Every $(f,K)$-transversal in $X(K)$ is dense in $X$ if and only if every $(f',K)$-transversal in $X'(K)$ is dense in $X'$.
    \end{enumerate}
\end{lemma}
\begin{proof}
    To prove part (a), let $P'\in X'(K)$ such that $\overline{\Orbit_{f'}(P')}=X'$.
    Then since $\f\circ f'=f\circ \f$, we find $\f(\Orbit_{f'}(P'))=\Orbit_{f}(\f(P'))$, and $\f$ being continuous and closed implies that
    \[\overline{\Orbit_{f}(\f(P'))}=\overline{\f(\Orbit_{f'}(P'))}=\f(\overline{\Orbit_{f'}(P')})=\f\left(X'\right)=X.\]
    Therefore $\f(P)$ has dense orbit in $X$.
    
    Now let $P\in X(K)$ such that $\overline{\Orbit_f(P)}=X$.
    Choose any $P'\in \f^{-1}(P)$.
    We know that $\f(\Orbit_{f'}(P'))=\Orbit_f(P)$, and 
    \[\f(\overline{\Orbit_{f'}(P')})=\overline{\f(\Orbit_{f'}(P'))}=\overline{\Orbit_{f}(P)}=X.\]
    Since $\f$ is a morphism, and it cannot increase the dimension of a variety, we find
    \[\dim\left(X'\right)=\dim(X)=\dim(\f(\overline{\Orbit_{f'}(P')}))\le \dim(\overline{\Orbit_{f'}(P')})\le \dim\left(X'\right).\]
    Therefore $\overline{\Orbit_{f'}(P')}=X'$, and $P'$ has dense orbit in $X'$.
    
    To prove part (b), let $S$ be a dense $(f,K)$-transversal in $X(K)$. 
    For each $s\in S$, choose one $s'\in \f^{-1}(s)$ to form the set $S'$:
    \[S'=\bigcup_{s\in S}\{s'\}\subset X'.\]
    Note that $S'$ is dense in $X'$, because 
    \[\dim\left(X'\right)=\dim(X)=\dim(\f(\overline{S'}))\le \dim(\overline{S'})\le \dim\left(X'\right).\]
    We claim that the points in $S'$ are in distinct grand $(f',K)$-orbits. 
    Suppose for contradiction that there exists $s_1,s_2\in S'$ such that $f'^n(s_1)=f'^m(s_2)$ for some $n,m$.
    Then 
    \[f^n(\f(s_1))=\f(f'^n(s_1))=\f(f'^m(s_2))=f^m(\f(s_2)),\] contradicting the assumption that $\f(s_1)$ and $\f(s_2)$ are in distinct grand $(f,K)$-orbits.
    Hence we may complete $S'$ to an $(f',K)$-transversal in $X'(K)$, which will still be dense.
    
    Now let $S'$ be a dense an $(f',K)$-transversal in $X'(K)$.
    By Lemma \ref{unionorbits}, the set $\f(S')$ can be partitioned into $\#\ker(\f)_{\Per}$ sets $S_1,\ldots,S_{\#\ker(\f)_{\Per}}$ such that for each $i$, the points in $S_i$ are in distinct grand $(f,K)$-orbits.
    Then 
    \[\overline{S_1}\cup\cdots\cup \overline{S_{\#\ker(\f)_{\Per}}}=\overline{S_1\cup\cdots\cup S_{\#\ker(\f)_{\Per}}}=\overline{\f(S')}=X,\] and by irreducibility of $X$, there exists $S_j$ such that $\overline{S_j}=X$.
    Therefore we may complete $S_j$ to an $(f,K)$-transversal in $X(K)$, which will still be dense.

    To prove part (c), we proceed by proving the contrapositive. 
    Let $S\subset X(K)$ be an $(f,K)$-transversal such that $\overline{S}\subsetneq X$.
    Observe that $\f^{-1}(S)$ is not dense in $X'$, because if it were, then \[\overline{S}=\overline{\f(\f^{-1}(S))}=\f(\overline{\f^{-1}(S)})=X,\] which contradicts the assumption that $S$ is not dense.
    
    Note that $\f^{-1}(S)$ contains an $(f',K)$-transversal:
    For any point $P'\in X'(K)$, there is a point $Q\in S$ such that $f^n(Q)=f^m(\f(P'))$ for some $n,m\in \ZZ$. 
    Let $Q_0\in \f^{-1}(Q)$, then 
    \[f^n(\f(Q_0))=\f(f'^n(Q_0))=\f(f'^m(P')),\]
    which implies that for some $\k'\in K$, we have $f'^n(Q_0)=f'^m(P')+\k'$.
    Decompose $f'$ into the composition of a homomorphism $\t'$ followed by a translation by $c'$ by the rigidity theorem:
    \[f'(x)=\t'(x)+c'.\]
    By equation (\ref{fnxa}) in Lemma \ref{unionorbits}, for any $N\in \NN$,
    \[f'^{n+N}(Q_0)=f'^{m+N}(P')+\t'^N(\k').\]
    Choose $N$ sufficiently large so that $\ker(\f)_{\Per}:=\t'^N(\ker(\f))=\t'^{N+1}(\ker(\f))$ stabilizes,
    and $\t'$ acts on $\ker(\f)_{\Per}$ periodically.
    Then $\t'^N(\k')\in \ker(\f)_{\Per}$ and $\t'$ is an automorphism on $\ker(\f)_{\Per}$.
    Take $\k\in \ker(\f)_{\Per}$ such that $\t'^{n+N}(\k)=-\t'^N(\k')\in \ker(\f)_{\Per}$, then we find
    \begin{align*}
        f'^{n+N}(Q_0+\k)=&f'^{n+N}(Q_0)+\t'^{n+N}(\k)\\
        =&f'^{n+N}(Q_0)-\t'^N(\k')\\
        =&f'^{m+N}(P')+\t'^N(\k')-\t'^N(\k')\\
        =&f'^{m+N}(P').
    \end{align*}

    We have found an element $Q_0+\k\in \f^{-1}(Q)$ in the same grand $(f',K)$-orbit as $P'$, so $\f^{-1}(S)$ contains an $(f',K)$-transversal.
    Take a subset $S'\subseteq\f^{-1}(S)$ that is an $(f',K)$-transversal, then ${S'}$ is not dense in $X'$ because $\f^{-1}(S)$ is not.

    Now let $S'\subset X'(K)$ be an $(f',K)$-transversal such that $\overline{S'}\subsetneq X'$.
    Then $\f(S')$ contains an $(f,K)$-transversal:
    For any point $P\in X$, let $P'\in \f^{-1}(P)$.
    Then there is $Q'\in S'$ such that $f'^n(Q')=f'^m(P')$, which implies
    \[f^n(\f(Q'))=\f(f'^n(Q'))=\f(f'^m(P'))=f^m(\f(P'))=f^m(P).\]
    Furthermore, the set $\f(S')$ cannot be dense in $X$, because 
    \[\dim(\overline{\f(S')})\le \dim (\overline{S'})<\dim\left(X'\right)=\dim(X).\]
    Take a subset $S\subseteq \f(S')$ that is an $(f,K)$-transversal, then $S$ is not dense in $X$.
\end{proof}

\begin{lemma}\label{iterinv}
    Let $K/\mathbb{Q}$ be a number field, $X/K$ an abelian variety. Let $f:X\to X$ be a regular self-map.
    Then for all $n\ge 1$:
    \begin{enumerate}[label=\emph{(\alph*)}]
        \item There is a point $P\in X(K)$ such that $\Orbit_f(P)$ is dense in $X$ if and only if there is a point $P'\in X(K)$ such that $\Orbit_{f^n}(P')$ is dense in $X$.
        \item There is a dense $(f,K)$-transversal in $X(K)$ if and only if there is a dense $(f^n,K)$-transversal in $X(K)$.
        \item Every $(f,K)$-transversal in $X(K)$ is dense in $X$ if and only if every $(f^n,K)$-transversal in $X(K)$ is dense in $X$.
    \end{enumerate}
\end{lemma}
\begin{proof}
Note that for any $P\in X(K)$, the $f$-orbit, $f^n$-orbit, grand $(f,K)$-orbit, and grand $(f^n,K)$-orbit of $P$ satisfy the following relations:
\begin{align}
    \Orbit_f(P)=&\bigcup_{i=0}^{n-1}\Orbit_{f^n}(f^i(P)),\label{union1}\\
    \grano_{f,K}(P)=&\bigcup_{i=0}^{n-1}\grano_{f^n,K}(f^i(P)).\label{union2}
\end{align}
The right-hand side is contained in the left-hand side in a straightforward way.
To prove the other containment, we let $Q\in \Orbit_f(P)$, then $Q=f^k(P)$ for some $k\ge 0$. 
Let $j=k\mod{n}$ and write $k=mn+j$, then $Q=(f^n)^m(f^j(P))$ and $Q\in \Orbit_{f^n}(f^j(P))$, proving equality (\ref{union1}).

Let $Q'\in \grano_{f,K}(P)$, then $f^h(Q)=f^k(P)$ for some $h,k\ge 0$. 
Let $j=(k-h )\mod{n}$, and we can write $h=k-j+mn$ for some integer $m$.
Take $a$ sufficiently large so that $an\ge h$ and let $b=a-m$, then $h=k-j+(a-b)n$.
Rearranging, we get $an-h=bn+j-k$
Then 
\[(f^n)^a(Q')=f^{an-h}(f^h(Q'))=f^{bn+j-k}(f^k(P))=(f^n)^b(f^j(P)),\] and $Q'\in \grano_{f^n,K}(f^j(P))$, proving equality (\ref{union2}).

To prove (a), if $\Orbit_{f^n}(P)$ is dense in $X$, then since $\Orbit_{f}(P)$ contains $\Orbit_{f^n}(P)$, it's also dense.
On the other hand, if $\Orbit_{f}(P)$ is dense, then 
\[\overline{\bigcup_{i=0}^{n-1}\Orbit_{f^n}(f^i(P))}=\bigcup_{i=0}^{n-1}\overline{\Orbit_{f^n}(f^i(P))}=X,\]
and by irreducibility of $X$, there exists $0\le i\le n-1$ such that $\overline{\Orbit_{f^n}(f^i(P))}=X$, as desired.

To prove (b), let $S\subset X(K)$ be a dense $(f,K)$-transversal in $X$.
Since any grand $(f,K)$-orbit is a union of grand $(f^n,K)$-orbits, we may complete $S$ to an $(f^n,K)$-transversal in $X$, which will still be dense.

Conversely, let $S'\subset X(K)$ be a dense $(f^n,K)$-transversal in $X$.
Since any grand $(f,K)$-orbit is a union of at most $n$ grand $(f^n,K)$-orbits, at most $n$ points of $S'$ can be in the same grand $(f,K)$-orbit.
Therefore we may partition $S'$ into $n$ sets $S_1,\ldots,S_n$ such that for a fixed $i$, all points in $S_i$ are in distinct grand orbits.
Now observe that 
\[\overline{S_1\cup\cdots \cup S_n}=\overline{S_1}\cup\cdots \cup \overline{S_n}=X,\] so by irreducibility of $X$, there is an $i$ such that $\overline{S_i}=X$.
Complete $S_i$ to an $(f,K)$-transversal, which will still be dense.

To prove (c), let $S'\subset X(K)$ be an $(f^n,K)$-transversal in $X$ such that $\overline{S'}$ is a proper subvariety of $X$. 
$S'$ contains an $(f,K)$-transversal, because every grand $(f,K)$-orbit is a union of grand $(f^n,K)$-orbits, which has representatives in $S'$.
Therefore we may take a subset $S\subset S'$ that is an $(f,K)$-transversal, and $S$ is not dense in $X$.

Conversely, let $S\subset X(K)$ be an $(f,K)$-transversal in $X$ such that $\overline{S}$ is a proper subvariety of $X$.
Then $\cup_{i=0}^{n-1}f^i(S)$ will contain an $(f^n,K)$-transversal $S'\subset \cup_{i=0}^{n-1}f^i(S)$.
Note that 
\[\overline{\cup_{i=0}^{n-1}f^i(S)}=\cup_{i=0}^{n-1}\overline{f^i(S)}=\cup_{i=0}^{n-1}f^i(\overline{S})\] cannot equal $X$, because $\overline{S}$ is a proper subvariety and $X$ is irreducible.
Then $\overline{S'}\ne X$ and $S'$ is not dense.
\end{proof}

\begin{lemma}\label{splitmap}
    Let $X$ be an abelian variety isogenous to a product
    $B_1^{n_1}\times\cdots\times B_k^{n_k}$ where the $B_i$ are simple and pairwise non-isogenous. 
    Let $A_i$ be the isotypic component\footnote{The unique maximal abelian subvariety of $A$ whose simple factors are all isogenous to $B_i$.} in $X$ corresponding to $B_i$, then we have $A_1+\cdots+A_k=X$ and an isogeny given by summation of the coordinates
    \[\f:\prod_{i=1}^k A_i\to X.\]
    
    Let $f:X\to X$ be a regular self-map on $X$, then we can decompose $f$ into $f=f_1+\cdots +f_k$ (not necessarily uniquely), where $f_i:A_i\to A_i$ is the restriction of $f$ to $A_i$.
    More precisely, for any $x\in X$, we can write $x=x_1+\cdots+x_k$ (not necessarily uniquely), where $x_i\in A_i$, and we have
    \[f(x)=f_1(x_1)+\cdots+f_k(x_k).\]
    Moreover, let map $f':\prod_{i=1}^k A_i\to \prod_{i=1}^k A_i$ be given by
    \[f'(x_1,\ldots,x_k)= (f_1(x_1),\ldots,f_k(x_k)),\]
    then the following diagram commutes:
    \begin{equation}\label{prodliftdiag}
        \begin{tikzcd}
        \prod_{i=1}^k A_i \arrow[r, "f'"] \arrow[d, "\f"] & \prod_{i=1}^k A_i \arrow[d, "\f"] \\
        X \arrow[r, "f"]                                   & X                                 
\end{tikzcd}
    \end{equation}
\end{lemma}
\begin{proof}
    The decomposition of $X$ into isotypic components is essentially the Poincar\'e reducibility theorem. We will prove the properties of $f:X\to X$.
    
    Since every map between abelian varieties is the composition of a homomorphism and a translation, we can write $f$ as 
    \[f(x)=\t(x)+c,\]
    where $\t$ is a homomorphism and $c\in X$.
    Consider the restriction of $\t$ on $A_i$:
    \[\t_i:=\t|_{A_i}:A_i\to A,\]
    we claim that the image of $\t_i$ is contained in $A_i$, i.e., $\t_i:A_i\to A_i$.
    To see that, consider the composition
    \[q\circ \t_i:A_i\to A\to A/A_i,\]
    where $q$ is the quotient-by-$A_i$ map.
    Since $A_i$ and $A/A_i$ have no isogenous factors, we know that $\Hom(A_i,A/A_i)=0$, and thus that $q\circ \t_i=0$, which means $\t_i(A_i)\subset \ker(q)=A_i$, proving the claim.

    Let $c=c_1+\cdots +c_k$ be a choice of decomposition such that $c_i\in A_i$ and define $f_i:A_i\to A_i$ given by $f_i(x)=\t_i(x)+c_i$, then for any choice of decomposition of $x=x_1+\cdots +x_k$, where $x_i\in A_i$, we find
    \begin{align*}
        \sum_{i=1}^k f_i(x_i)&=\sum_{i=1}^k (\t_i(x_i)+c_i)\\
        &=\left(\sum_{i=1}^k \t_i(x_i)\right)+c\\
        &=\left(\sum_{i=1}^k \t(x_i)\right)+c\\
        &= \t\left(\sum_{i=1}^kx_i\right)+c\\
        &= \t\left(x\right)+c\\
        &=f(x).
    \end{align*}
    It then follows that diagram (\ref{prodliftdiag}) commutes, as $\f$ is effectively the summation map.
\end{proof}

\begin{lemma}\label{projorbit}
    Let $X_1,\ldots,X_k$ be sets, let $f_i:X_i\to X_i$ be maps, let $X\cong X_1\times \cdots \times X_k$ be the product, and define $f:X\to X$ by $f(x_1,\ldots,x_k)=(f_1(x_1),\ldots,f_k(x_k))$. Let $x=(x_1,\ldots,x_k)$ and $y=(y_1,\ldots,y_k)$ be points of $X$. If there exists an index $i$ such that $\Orbit_{f_i}^\grand(x_i)\ne\Orbit_{f_i}^\grand(y_i)$, then $\Orbit_{f}^\grand(x)\ne\Orbit_{f}^\grand(y)$.
\end{lemma}
\begin{proof}
    Let $x=(x_1,\ldots,x_k)$ and $y=(y_1,\ldots,y_k)$ be points of $X$ such that $x$ and $y$ are in the same grand $(f,K)$-orbit, then there exists integers $n,m$ such that $f^n(x)=f^m(y)$.
    By the decomposition of $f$, we get that 
    \[(f_1^n(x_1),\ldots,f_k^n(x_k))=(f_1^m(y_1),\ldots,f_k^m(y_k)),\]
    so $x_i\in \grano_{f_i,K}(y_i)$ for all $i$, proving the contrapositive of the lemma.
\end{proof}
\begin{remark}
    Note that the converse of Lemma \ref{projorbit} is usually not true. 
    For a concrete counterexample, consider $X=\ZZ\times \ZZ$, $f:X\to X$ given by $f(x,y)=(2x,3y)$, and take $(1,1),(1,3)\in X$. 
    Observe that $1$ is in the same grand $f_1$-orbit as $1$, and $1$ is in the same grand $f_2$-orbit as $3$.
    However, the grand $(f,K)$-orbits of $(1,1)$ and $(1,3)$ are disjoint:
    \begin{align*}
        \grano_{f,K}((1,1))=&\{(2^n,3^n)\in X\mid n\in \ZZ\},\\
        \grano_{f,K}((1,3))=&\{(2^n,3^{n+1})\in X\mid n\in \ZZ\}.
    \end{align*}
    In general, the set of grand orbits of a map on a product is not equal to the product of the sets of grand orbits of the component maps, precisely because of the time synchronization issue demonstrated in the example.
\end{remark}

\section{Density of special sets in abelian varieties}
Goursat's lemma in group theory \cite[Theorem 4]{goursat} states that the subgroups of $G\times G'$ that surject to both factors under the projection maps are in bijection with triples consisting of $(N,N',g)$, where $N\trianglelefteq G$ and $N'\trianglelefteq G'$ are normal subgroups and $g:G/N\xrightarrow{\sim} G'/N'$ is an isomorphism.
The lemma holds in the isogeny category of abelian varieties as well, and we will use it to find dense sets in abelian varieties.

The isogeny category of abelian varieties over a field $K$ is a category with the abelian varieties over $K$ as objects and 
\[\Hom^0(A,A'):=\Hom(A,A')\otimes_\ZZ\QQ,\quad \End^0(A):=\End(A)\otimes_{\ZZ}\QQ\]
as morphism groups.
In the isogeny category, isogenies are invertible and hence isomorphisms in the category.
\begin{lemma}\label{goursatab}
    Let $A$ and $A'$ be abelian varieties over a field $K$. Then abelian subvarieties $H\subseteq A\times A'$ that surject to both factors under the projection maps are in bijection with triples consisting of $(N,N',\g)$, where $N\subseteq A$ and $N'\subseteq A'$ are abelian subvarieties and $\g:A/N\to A'/N'$ is a quasi-isogeny, i.e., an isomorphism in the isogeny category.
\end{lemma}
\begin{proof}
    We will follow the proof structure in \cite{goursat}, constructing maps in both directions, and verifying that they are inverses to each other.

    Let $H\subseteq A\times A'$ be an abelian subvariety that surjects to both factors under the projection maps.
    Let 
    \[N=(H\cap (A\times 0))^\circ, \quad N'=(H\cap (0\times A'))^\circ\]
    be the identity components of the intersections respectively, and let 
    \[q:A\to A/N,\quad q':A'\to A'/N'\] be the quotient maps.
    The image of $H$ under the component-wise quotient map, $\overline{H}=(q, q')(H)$, is an abelian subvariety of $A/N\times A'/N'$.
    Note that $\overline{H}$ surjects to both $A/N$ and $A'/N'$ under the projection maps.
    
    We claim that the projection maps 
    \[p:\overline{H}\to A/N, \quad p':\overline{H}\to A'/N'\] are in fact isogenies.
    Let $(0,y)\in \ker(p)$, then there exists $(a,b)\in H$ such that $q(a)=0$ and $q'(b)=y$.
    Note that $(a,0)\in N\subseteq H$, and $(0,b)=(a,b)-(a,0)\in H$ as well, which implies that $(0,b)\in (H\cap (0\times A'))$.
    Then we find $(0,y)=(q, q')(0,b)\in (H\cap (0\times A'))/N'$, so $\ker(p)\subseteq (H\cap (0\times A'))/N'$.
    The set $(H\cap (0\times A'))/N'$ is finite because a closed algebraic subgroup is a finite union of translates of its identity component.
    Therefore $\ker(p)$ is finite, and the argument for $p'$ is symmetric.

    Since $p$ and $p'$ are isogenies, they are isomorphisms in the isogeny category, and we may take 
    \[\g=p'\circ p^{-1}\in \Hom^0(A/N,A'/N'),\] 
    which finishes the construction of the triple $(N,N',\g)$.

    Now let $(N,N',\g)$ be a triple where $N\subseteq A$ and $N'\subseteq A'$ are abelian subvarieties and $\g\in \Hom^0(A/N,A'/N')$ is a quasi-isogeny.
    Let $m>0$ be an integer such that $m\g\in \Hom(A/N,A'/N')$.
    Then consider the graph $\overline{H}$ of the quasi-isogeny $\g$ in $A/N\times A'/N'$, which is defined as the image of $A/N$ under the map $([m], m\g):A/N\to A/N\times A'/N'$:
    \begin{align*}
        \overline{H}=&([m],m\g)(A/N)\\
        =&\{(ma,(m\g)(a))\in A/N\times A'/N'\mid a\in A/N\}.
    \end{align*}
    Note that the set $\overline{H}$ is independent of the choice of $m$.
    For any other choice $m'$, we can compare $([m],m\g)(A/N)$ and $([m'],m'\g)(A/N)$ to $([\lcm(m,m')],\lcm(m,m')\g)(A/N)$ and find they are all equal, as multiplication by a nonzero integer is an isogeny and therefore surjective. 
    By construction, the set $\overline{H}$ is an abelian subvariety of $A/N\times A'/N'$ and surjects to both $A/N$ and $A'/N'$ under the projection maps.
    Take 
    \[H=(q, q')^{-1}(\overline{H}).\] 
    It is a closed algebraic subgroup of $A\times A'$. 
    In fact, it is an abelian subvariety of $A\times A'$ because of the following short exact sequence:
    \[0\to N\times N'\to H\to \overline{H}\to 0,\]
    where $N\times N'$ and $\overline{H}$ are both connected.
    
    We claim that $H$ surjects to both $A$ and $A'$ under the projection maps.
    Let $a\in A$.
    Since $[m]:A/N\to A/N$ is surjective, we can choose $x\in A/N$ such that $mx=q(a)$.
    Choose $b\in A'$ such that $q'(b)=m\g(x)$, then 
    \[
        (q(a),q'(b))=(mx,(m\gamma)(x))\in \overline{H},
    \]
    so $(a,b)\in H$. Thus $H$ surjects to $A$.
    Similarly, let $b\in A'$. 
    Since $m\g$ is an isogeny and therefore surjective, there exists $x\in A/N$ such that $m\g(x)=q'(b)$.
    Then we can choose $a\in A$ such that $q(a)=mx$, which implies that \[(q(a),q'(b))=(mx,(m\g)(x))\in \overline{H},\]
    so $(a,b)\in H$, and we have shown that $H$ surjects to $A'$ as well.

    To show the two constructions are inverses to each other, we first start with $H\subseteq A\times A'$ an abelian subvariety that surjects to both factors under the projection maps. Then construct
    \[
        N,\quad N',\quad \overline{H},\quad p:\overline{H}\to A/N,\quad
        p':\overline{H}\to A'/N',\quad \text{and }\gamma=p'\circ p^{-1}.
    \]
    To reconstruct $H$, we first take the graph $\G$ of $\g$: let $m>0$ such that $mp^{-1}$ and $m\g$ are isogenies, then
    \[\G=\{(ma,m\g(a))\in A/N\times A'/N'\mid a\in A/N\}.\]
    Observe that $\overline{H}=\G$:
    For every point $(a,b)\in \overline{H}$, there exists some $(x,y)\in \overline{H}$ such that $(a,b)=(mx,my)$, because $[m]$ is surjective on $\overline{H}$.
    We find 
    \[my=p'\circ p^{-1}(mx)=m\g(x),\] 
    so $(a,b)=(mx,m\g(x))\in \G$ and $\overline{H}\subseteq \G$.
    For every point $(ma,m\g(a))\in \G$ where $a\in A/N$, we find 
    \[(ma,m\g(a))=(p(mp^{-1}(a)),p'(mp^{-1}(a))),\]
    where $mp^{-1}(a)\in \overline{H}$ and $(ma,m\g(a))$ is the point $mp^{-1}(a)$ written in coordinates, so $\G\subseteq\overline{H}$.
    
    We claim that $H=(q, q')^{-1}(\overline{H})$.
    It's clear that $H\subseteq (q, q')^{-1}(\overline{H})$.
    Let $(a,b)\in (q, q')^{-1}(\overline{H})$, then there exists $(a_0,b_0)\in H$ such that $q(a_0)=q(a)$ and $q'(b_0)=q'(b)$.
    Since $q$ and $q'$ are homomorphisms, we find that $a-a_0\in N$ and $b-b_0\in N'$, and therefore $(a-a_0,0),(0,b-b_0)\in H$.
    Then 
    \[(a,b)=(a_0,b_0)+(a-a_0,0)+(0,b-b_0)\in H,\]
    as desired.

    For the other direction, let $(N,N',\g)$ be a triple where $N\subseteq A$ and $N'\subseteq A'$ are abelian subvarieties and $\g\in \Hom^0(A/N,A'/N')$ is a quasi-isogeny.
    Then let $m>0$ such that $m\g$ is an isogeny, and we get
    \begin{gather*}
        \overline{H}=([m],m\g)(A/N)=\{(ma,m\g(a))\in A/N\times A'/N'\mid a\in A/N\},\\
        q:A\to A/N,\quad q':A'\to A'/N',\\
        H=(q, q')^{-1}(\overline{H}).
    \end{gather*}
    To reconstruct the triple from $H$, we take
    \[M=(H\cap (A\times 0))^\circ, \quad M'=(H\cap (0\times A'))^\circ.\]
    It's clear from construction that $N\subseteq M$ and $N'\subseteq M'$.
    We claim that equalities in fact hold.
    Note that $q(H\cap (A\times 0))\subseteq \overline{H}\cap (A/N\times 0)$, which is a finite set because $\overline{H}\cap (A/N\times 0)$ is contained in the kernel of the isogeny $p':\overline{H}\to A'/N'$.
    Since $q(M)$ is connected and contained in a finite set, we get $q(M)=0$, and $M\subseteq N$.
    The argument for $M'=N'$ is symmetric.
    Finally, note that 
    \[p\circ ([m],m\g)=[m]:A/N\to A/N,\quad  p'\circ ([m],m\g)=m\g:A/N\to A'/N'.\]
    Then in the isogeny category, we have
    \[p'\circ p^{-1}=(p'\circ ([m],m\g))\circ(p\circ ([m],m\g))^{-1}=m\g \circ [m]^{-1}=\g,\]
    and we recover $\g$ as well.
\end{proof}

Now we prove various lemmas that help us with  demonstrating that sets are Zariski dense in $X$ and constructing dense sets in $X$.

\begin{lemma}\label{densityreduction}
Let $K$ be a field of characteristic 0 and let $X$ be an abelian variety over $K$. 
Let $\f$ be an isogeny from $X$ to a product of isotypic components of $X$
\[\f:X\to \prod_{i=1}^nA_i,\]
where $A_i\subseteq X$ is the maximal abelian subvariety that is isogenous to a power of a simple abelian variety $B_i$, and $B_i$ is not isogenous to $B_j$ for $i\ne j$.
Such isogenies exist by the Poincar\'e reducibility theorem.
Composing $\f$ with the projection map $p_i':\prod_{i=1}^nA_i\to A_i$, we get
\[p_i=p_i'\circ \f:X\to A_i.\]
Let $S\subseteq X(K)$ be a subset of $X(K)$. 
Suppose that any one of the following three conditions is true:
\begin{enumerate}[label=\emph{(\alph*)}]
    \item $S$ is contained in a finitely generated subgroup $\Gamma\subseteq X(K)$ and $\overline{S}$ is irreducible, or
    \item $S=\G\subseteq X(K)$ is a finitely generated subgroup of $X(K)$, or
    \item $S=\Orbit_f(P)$ is the orbit of a point $P\in X(K)$ under a dominant self-map $f$ of $X$, 
\end{enumerate}
then $S$ is Zariski dense in $X$ if and only if its image under $p_i$ is Zariski dense in $A_i$ for all $1\le i\le n$.
\end{lemma}
\begin{proof}
    Since an isogeny preserves Zariski density, the set $S$ is dense in $X$ if and only if $\f(S)$ is dense in $\prod_{i=1}^nA_i$. 
    Being in a finitely generated subgroup, having irreducible closure, being a finitely generated subgroup, and having dense image under $p_i$ are all preserved under an isogeny as well.
    Henceforth, we replace $X$, $\G$, and $S$ with $\prod_{i=1}^nA_i$, $\f(\G)$, and $\f(S)$, respectively, and we assume that $X\cong\prod_{i=1}^nA_i$ in the proof of part (a) and (b).

    We begin with part (a). 
    Let $S$ be a subset of $X(K)$ contained in a finitely generated subgroup $\Gamma\subseteq X(K)$ such that $\overline{S}$ is irreducible.
    If $\overline{S}=X$, then since the projection map $p_i$ is closed for all $i$, we have $\overline{p_i(S)}=p_i(\overline{S})=A_i$, and the image of $S$ is dense in all $A_i$.

    Conversely, suppose that $\overline{p_i(S)}=A_i$ for all $i$. 
    By Faltings's theorem \cite{faltings}, the set $\overline{S}\cap \G$ is a finite union of translates of subgroups of $\G$, and $S\subset \overline{S}\cap \G\subset\overline{S}$.
    Hence $\overline{S}=\overline{(\overline{S}\cap \G)}$, which is a union of translates of abelian subvarieties of $X$.
    
    In particular, since $\overline{S}$ is irreducible, it is a translate of an abelian subvariety of $X$ which maps surjectively under all $p_i$. 
    Note that the image of a translate under an endomorphism of abelian varieties is a translate of the image, hence the abelian subvariety also maps surjectively under all projections.
    We now show that the only abelian subvariety of $X$ that maps surjectively under $p_i$ for all $i$ is $X$ itself.
    
    For $n=1$, the product has 1 component and $\prod_{i=1}^nA_i=A_1$, and the projection map $p_1$ is the identity map. 
    The only abelian subvariety of $A_1$ that maps surjectively under $p_1$ is $X=A_1$.
    For $n>1$, we have
    \[\prod_{i=1}^nA_i=\left(\prod_{i=1}^{n-1}A_i\right)\times A_{n}.\]
    Let $B\subseteq \prod_{i=1}^nA_i$ be an abelian subvariety that maps surjectively under all projections.
    By the inductive hypothesis, the subvariety $B$ maps surjectively under the projection to $\prod_{i=1}^{n-1}A_i$ as well.
    Since $\prod_{i=1}^{n-1}A_i$ and $A_{n}$ have no isogenous factors, by Lemma \ref{goursatab}, the only possibility for the triple corresponding to $B$ is $(N,N',\g)$ where $N=\prod_{i=1}^{n-1}A_i$, $N'=A_n$, and the quasi-isogeny $\g\in \Hom^0((\prod_{i=1}^{n-1}A_i)/N, A_n/N')$ is trivial:
    \[\g:\left.\prod_{i=1}^{n-1}A_i\middle/\left(\prod_{i=1}^{n-1}A_i\right)\right.\to A_{n}/A_{n}.\] 
    Hence $B=\prod_{i=1}^nA_i$, and $\overline{S}= X$.
    This concludes the proof of (a)
    
    For part (b), let $S=\G$ be a finitely generated subgroup of $X(K)$.
    If $\G$ is dense, then since $p_i$ is closed, we find $\overline{p_i(\G)}=p_i(\overline{\G})=A_i$ for all $i$.
    
    Conversely, suppose that $\overline{p_i(\G)}=A_i$ for all $i$. 
    Let $G=\overline{\G}$ be the Zariski closure of $\G$. 
    $G$ is a closed algebraic subgroup of $X$.
    Let $G^\circ$ be its identity component, then $G^\circ$ is an abelian subvariety, and there exists $P_1,\ldots,P_k\in G$ such that
    \[G=G^\circ\cup (P_1+G^\circ)\cup \cdots \cup (P_k+G^\circ).\]
    Since $A_i$ is irreducible, we deduce that the abelian subvariety $G^\circ$ maps surjectively to all factors under $p_i$, and by the same argument as in part (a), we find $G^\circ=X$ and $\G$ is dense in $X$.

    For part (c), if $\overline{\mathcal{O}_f(P)}=X$, then since $p_i$ is closed, we have $\overline{p_i(\mathcal{O}_f(P))}=p_i(\overline{\mathcal{O}_f(P)})=A_i$, and the image of $\Orbit_f(P)$ under $p_i$ is dense in $A_i$ for all $i$.

Conversely, suppose that $\overline{p_i(\mathcal{O}_f(P))}=A_i$ for all $i$. 
Let $\overline{\mathcal{O}_f(P)}=Z$, then ${p_i(Z)}=A_i$ for all $i$.
Write $Z=\bigcup_{j=1}^k C_j$, where $C_j$ are the irreducible components of $Z$.
Note that $f$ permutes $\{C_j\mid 1\le j\le k\}$ transitively, for otherwise $Z$ would have been the invariant subset containing $P$.

We claim that all irreducible components project to all isotypic components surjectively, i.e., for all $i,j$, 
$p_i(C_j)=A_i$.

Let $1\le i\le n$.
Observe that $\bigcup_{j=1}^k p_i(C_j)=p_i(\bigcup_{j=1}^k C_j)=A_i$. 
Since $A_i$ is irreducible, there is a $C_\ell$ such that $p_i(C_\ell)=A_i$.
Furthermore, since $f$ splits as $f_1+\cdots+f_n$ where $f_i$ is the restriction of $f$ on $A_i$ by Lemma \ref{splitmap}, we find \[p_i(f(C_\ell))=f_i(p_i(C_\ell))=f_i(A_i)=A_i.\]
Then by transitivity of the action of $f$ on $\{C_j\mid 1\le j\le k\}$, all $C_j$ surject onto $A_i$.

By Fact 3.6 of \cite{ghiocascanlon}, the orbit $\mathcal{O}_f(P)$ is contained in a finitely generated subgroup $\Gamma\subset X$.
Then we may use part (a) on $C_1\cap \mathcal{O}_f(P)$ and conclude that $\overline{\mathcal{O}_f(P)}= X$.
\end{proof}

\begin{remark}\label{genus2}
    Note that both assumptions in part (a) of Lemma \ref{densityreduction} are needed for the statement to hold. 
    If we remove the assumption that $\overline{S}$ is irreducible, then for a product of non-isogenous elliptic curves $E_1\times E_2$, we can let $P_1\in E_1$ and $P_2\in E_2$,
    where 
    \[
    \overline{\{nP_1\mid n\in \ZZ\}}=E_1
    \quad\text{and}\quad 
    \overline{\{nP_2\mid n\in \ZZ\}}=E_2,
    \]
    and take 
    \[
    \Bigl(\{nP_1\mid n\in \ZZ\}\times \{0\}\Bigr)\cup \Bigl(\{0\}\times \{nP_2\mid n\in \ZZ\}\Bigr),
    \]
    which is contained in a finitely generated subgroup and has dense image under both projections but is not dense.

    On the other hand, if we remove the finite generation assumption, then for any genus 2 curve $C\subset E_1\times E_2$, the set $C(\Qbar)$ will be dense under both of the projection maps, as $C$ cannot live in any of the fibers and must surject to both elliptic curves.
    However, $C(\Qbar)$ is not dense in $E_1\times E_2$.
\end{remark}

\begin{lemma}\label{cyclicsubset}
    Let $K$ be a field of characteristic 0, $X/K$ an abelian variety, and $P\in X(K)$ a point. 
    Let $S=\{nP\mid n\in \ZZ\}$ be the subgroup generated by $P$ and suppose that $\overline{S}$ is connected.
    Then for any $T\subset S$ such that $\# T=\infty$, we have $\overline{T}=\overline{S}$.

    If $S$ is not necessarily connected, then there exists a point $c\in X(K)$ such that
    \[c+\overline{S}^\circ\subseteq \overline{T},\]
    where $\overline{S}^\circ$ is the identity component of $\overline{S}$.
\end{lemma}
\begin{proof}
    Since $S$ is a subgroup of $X$, its Zariski closure $\overline{S}$ is an algebraic group. 
    Furthermore, we have that $\overline{S}$ is connected, so $\overline{S}$ is an abelian subvariety.
    
    Consider the set $S\cap \overline{T}$. 
    By Falting's theorem, there exists $c_i\in S$ such that
    \[S\cap \overline{T}=\bigcup_{i=1}^k (c_i+S_i),\]
    where $S_i$ are subgroups of $S$.
    Since $T\subset S\cap \overline{T}$ is infinite, at least one of the $S_i$ is an infinite subgroup of $S$, which must be of the form $\{mnP\mid n\in\ZZ\}=[m]S$ for some non-zero integer $m$.
    Now observe that $\overline{[m]S}=[m]\overline{S}=\overline{S}$ because $\overline{S}$ is an abelian subvariety and $[m]$ is an isogeny, which is surjective and closed.
    Knowing that $c_j+[m]S\subset \overline{T}$ for some $j$ and $\overline{T}$ is closed, we deduce that $c_j+\overline{S}\subset \overline{T}\subset \overline{S}$, and therefore $c_j+\overline{S}=\overline{S}$ and $\overline{T}=\overline{S}$.

    Now we remove the assumption that $\overline{S}$ is connected.
    Since $S$ is a subgroup of $X$, its Zariski closure $\overline{S}$ is an algebraic group, hence of the form
    \[\overline{S}=\bigcup_{i=1}^k (P_i+\overline{S}^\circ),\]
    where $P_i\in \overline{S}$ and $\overline{S}^\circ$ is the identity component.
    Since $\#T=\infty$, there exists some $j$ such that $\#(T\cap (P_j+\overline{S}^\circ))=\infty$.
    Then we may apply the conclusion for the connected case to $T\cap (P_j+\overline{S}^\circ)-P_j$ and $S\cap \overline{S}^\circ$, which then implies 
    \[\overline{T\cap (P_j+\overline{S}^\circ)-P_j}= \overline{S}^\circ.\]
    Shifting by $P_j$, we obtain the desired containment
    \[P_j+\overline{S}^\circ =\overline{T\cap (P_j+\overline{S}^\circ)}\subseteq \overline{T}.\qedhere\]
\end{proof}

\begin{lemma}\label{indcoord}
    Let $B$ be a simple abelian variety over a field $K$ of characteristic~$0$.
    Let $\f: A\to B^n$ be an isogeny over $K$.
    
    Let $P\in A(K)$ and write $\f(P)=(P_1,\ldots,P_n)\in B^n(K)$. 
    If $\{P_1\otimes1,\ldots,P_n\otimes1\}$ is $\End^0(B)$-linearly independent in $B(K)\otimes_\ZZ \QQ$, then $\ZZ P=\{mP\mid m\in\ZZ\}$ is Zariski dense in $A$ and $\Ann(P)=0$ in $\End(A)$.
\end{lemma}
\begin{proof}
    Let $P\in A(K)$ and $\f(P)=(P_1,\ldots,P_n)\in B^n(K)$ such that
    $P_1\otimes1,\ldots,P_n\otimes1$ are $\End^0(B)$-linearly independent in $B(K)\otimes_\ZZ \QQ$.
    
    Let $Z$ be the Zariski closure of $\ZZ P$, and suppose for the sake of contradiction that $Z\subsetneq A$.
    The subvariety $Z$ is an algebraic subgroup of $A$, hence a finite disjoint union of translates of its identity component $Z^\circ$.
    Note that $Z^\circ$ is a proper abelian subvariety of $A$ because $Z$ is a proper subvariety of $A$.
    Since $P\in Z$, there exists an integer $k>0$ such that $kP\in Z^\circ$.
    
    Since $\f$ is an isogeny, we have $\f(Z^\circ)$ is a proper abelian subvariety of $B^n$ .
    Then by the Poincar\'e reducibility theorem, there exists a nonzero homomorphism $q:B^n\to B$ of abelian varieties such that $\f(Z^\circ) \subseteq \ker(q)$.
    We can extend scalars to $\QQ$ and use Poincar\'e to get that $\Hom^0(B^n,B)\cong\prod_{i=1}^n \End^0(B)$, so viewing $q$ in $\Hom^0(B^n,B)$, we have $q=q_1+\cdots +q_n$ where $q_i\in \End^0(B)$ and not all zero.
    Evaluating $q$ on $\f(kP)$, we get
    \begin{align*}
        q(\f(kP))=&q_1(kP_i\otimes1)+\cdots +q_n(kP_n\otimes1)\\
        =&kq_1(P_1\otimes1)+\cdots +kq_n(P_n\otimes1),
    \end{align*}
    which contradicts the points $P_1\otimes1,\ldots,P_n\otimes1$ being $\End^0(B)$-linearly independent.
    Therefore $Z=A$ and $\ZZ P$ is Zariski dense in $A$.

    Now we show that $P$ has trivial annihilator in $\End(A)$.
    Let $g\in \End(A)$ such that $g(P)=0$.
    Let $\f^\vee: B^n\to A$ be the isogeny dual to $\f$ such that $\f\circ \f^\vee = [k]_{B^n}$ and $\f^\vee\circ \f = [k]_{A}$ for some integer $k\ge 1$.
    Then $g'=\f\circ g\circ \f^\vee\in \End(B^n)$ satisfies that \[g'(\f(P))=\f\circ g\circ \f^\vee(\f(P))=\f(g(kP))=\f(kg(P))=\f(0)=0.\]
    Note that $\End(B^n)\cong \Mat_n(\End(B))$.
    Since the coordinates of $\f(P)$ are $\End^0(B)$-linearly independent by assumption, the equality $g'(\f(P))=0$ implies that the linear combinations at all coordinates are trivial.
    Therefore the matrix representing $g'$ is 0 and $g'=0$, which then implies that $g=0$ as $\f$ and $\f^\vee$ are isogenies.
\end{proof}

\section{Proof of the main theorem}

The last ingredient in the proof of the main theorem is a linear algebra fact on affine maps, as $X(K)$ is a finitely generated $\ZZ$-module and every regular self-map $f$ acts as an affine map on $X(K)$.

\begin{proposition}\label{unipq}
    Let $V$ be a finite dimensional $\QQ$-vector space and $\{P,Q\}$ be linearly independent vectors in $V$.
    Let $U=\id+N$ be a unipotent linear map where $N$ is a nilpotent linear map.
    Let $f:V\to V$ be an affine map of the form $f(x)=U(x)+Q$.
    Then at most $2$ multiples of $P$ can be in the same grand $(f,\QQ)$-orbit.
    That is, the map $n\mapsto \grano_{f,\QQ}(nP)$ is at most $2$-to-$1$.
\end{proposition}
\begin{proof}
    Suppose there exists distinct rational numbers $a,b,c$ and integers $n_a<n_b<n_c$ such that $f^{n_a}(aP)=f^{n_b}(bP)=f^{n_c}(cP)$.
    Since $f$ is a bijection, we may apply iterates of $f^{-1}$ to the chain of equalities and relabel them with $0<r=n_c-n_b<s=n_c-n_a$ to get 
    \[bP=f^{r}(cP),\quad aP=f^{s}(cP).\]
    
    We will construct a homogenization $F$ of $f$ that is linear and unipotent, and we will show that the images of $P$ under $3$ distinct iterates of $F$ cannot be trapped in the subspace corresponding to $\Span(P)$ in the homogenization.

    Let $\Tilde{V}=V\oplus e\QQ$ and define the linear map $F:\Tilde{V}\to \Tilde{V}$ by $F(v+te)=U(v)+tQ+te$ for any $v+te\in \Tilde{V}$.
    Note that $F(v+e)=f(v)+e$ for all $v\in V$. 
    In other words, the restriction of $F$ on the affine subspace of $V$ shifted by $e$ is precisely $f$.

    Let $\Tilde{N}:\Tilde{V}\to \Tilde{V}$ be given by
    \[\Tilde{N}=F-\id,\quad \Tilde{N}(v+te)=N(v)+tQ.\]
    Note that $\Tilde{N}$ is nilpotent, because 
    \[\Tilde{N}^m(v+te)=N^m(v)+tN^{m-1}(Q),\]
    which equals 0 for $m$ sufficiently large, as $N$ is nilpotent.

    Write $\Tilde{w}=cP+e$ and homogenize to get
    \begin{equation}\label{iterF}
        bP+e=F^{r}(\Tilde{w}),\quad aP+e=F^{s}(\Tilde{w}),
    \end{equation}
    which implies that \[\Tilde{w},F^{r}(\Tilde{w}),F^{s}(\Tilde{w})\in \Span(P,e),\]
    so they are linearly dependent.

    We claim that for any nilpotent map $M$ and vector $w$ such that $M^k(w)=0$ and $M^{k-1}(w)\ne 0$ for some positive integer $k$, the set $\{w, M(w),\ldots, M^{k-1}(w)\}$ is linearly independent.
    Suppose that \[\sum_{i=0}^{k-1} c_iM^i(w)=0\]
    is a nontrivial linear relation of the set.
    Let $n$ be the minimal index for nonzero coefficients in the linear relation and apply $M^{k-1-n}$ to the linear relation, then we get
    \[c_nM^{k-1}(w)=0,\]
    contradicting the assumption that $M^{k-1}(w)$ is nonzero.

    Next we claim that $\Tilde{N}^2(\Tilde{w})=0$. 
    Suppose not, then $\{\Tilde{w}, \Tilde{N}(\Tilde{w}),\Tilde{N}^2(\Tilde{w})\}$ is linearly independent by the claim.
    At the same time, we find 
    \[F^n(\Tilde{w})=(\Tilde{N}+\id)^n(\Tilde{w})=\sum_{i=0}^n{n\choose i}\Tilde{N}^i(\Tilde{w}).\]
    In particular, we compute: 
    \begin{align*}
        \Tilde{w}=&\Tilde{w},\\
        F^r(\Tilde{w})=&\Tilde{w}+r\Tilde{N}(\Tilde{w})+{r\choose 2}\Tilde{N}^2(\Tilde{w})+(\text{higher order terms}),\\
        F^s(\Tilde{w})=&\Tilde{w}+s\Tilde{N}(\Tilde{w})+{s\choose 2}\Tilde{N}^2(\Tilde{w})+(\text{higher order terms}),
    \end{align*}
    and the coefficients of $\Tilde{w}, \Tilde{N}(\Tilde{w}),\Tilde{N}^2(\Tilde{w})$ imply that the set $\{\Tilde{w},F^{r}(\Tilde{w}),F^{s}(\Tilde{w})\}$ is linearly independent, contradicting the earlier observation that they are in the two dimensional span $\Span(P,e)$.

    Therefore $\Tilde{N}^2(\Tilde{w})=0$. 
    Then \[F^n(\Tilde{w})=\Tilde{w}+n\Tilde{N}(\Tilde{w}),\]
    and we can substitute the equalities into equation (\ref{iterF}) to get
    \[bP+e=F^{r}(\Tilde{w})=(cP+e)+r(\Tilde{N}(\Tilde{w})),\]
    which implies that
    \begin{equation}\label{NwinP}
        \Tilde{N}(\Tilde{w})=\frac{b-c}{r}P.
    \end{equation}
    Apply $\Tilde{N}$ to both sides, we get 
    $0=\Tilde{N}(p)$, so $N(p)=0$.
    Now we recall the definition of $\Tilde{N}$ and find that
    \begin{equation}\label{NwinQ}
        \Tilde{N}(\Tilde{w})=\Tilde{N}(cP+e)=N(cP)+Q=Q.
    \end{equation}
    Combining equations (\ref{NwinP}) and (\ref{NwinQ}), we have found a nontrivial linear relation 
    \[\frac{b-c}{r}P=Q
    \]
    between $P$ and $Q$, which is a contradiction.
    Therefore there cannot be 3 distinct multiples of $P$ in the same grand $(f,\QQ)$-orbit.
\end{proof}

\begin{remark}
    The bound 2 is sharp in Proposition \ref{unipq}, as illustrated by the following example.
    
    Let $V=\QQ^2$, $P=(1,0)$, $Q=(0,1)$, and 
    \[N=\begin{bmatrix}
        1 & 1\\
        -1& -1
    \end{bmatrix}.\]
    Then we find
    \[f(P)=(\id+N)(P)+Q=(1,0)+(1,-1)+(0,1)=(2,0)=2P.\]
\end{remark}

\begin{corollary}\label{unipz}
    Let $M$ be a finitely generated $\ZZ$-module and $\{P,Q\}$ be $\ZZ$-linearly independent.
    Let $f:M\to M$ be a map of the form $f(x)=U(x)+Q$, where $U=\id+N$ is a unipotent $\ZZ$-linear map with $N$ a nilpotent $\ZZ$-linear map and $Q$ is non-torsion.
    Then infinitely many multiples of $P$ are in distinct grand $(f,\ZZ)$-orbits.
\end{corollary}
\begin{proof}
    Let $f_\QQ:M\otimes_\ZZ \QQ\to M\otimes_\ZZ \QQ$ be given by $f_\QQ(x)=U_\QQ(x)+Q\otimes 1$, where $U_\QQ:M\otimes_\ZZ \QQ\to M\otimes_\ZZ \QQ$ is the natural $\QQ$-linear map obtained from $U$.
    Since $Q$ is non-torsion, we know $Q\otimes1$ is nonzero.
    Note that the following diagram commutes:
    \begin{center}
        \begin{tikzcd}
        M \arrow[d, "\id\otimes 1"] \arrow[r, "f"] & M \arrow[d, "\id\otimes 1"] \\
        M\otimes_\ZZ \QQ \arrow[r, "f_\QQ"]        & M\otimes_\ZZ \QQ           
        \end{tikzcd}
    \end{center}
    Let $x\in M$, then 
    \begin{align*}
        \id\otimes 1(f(x))&=(U(x)+Q)\otimes 1 = U(x)\otimes 1+Q\otimes 1=U_\QQ(x\otimes 1)+Q\otimes 1,\\
        f_\QQ(\id\otimes 1(x))&=f_\QQ(x\otimes 1)=U_\QQ(x\otimes 1)+Q\otimes 1.
    \end{align*}
    If follows that if $x_1,x_2\in M$ are in the same grand $(f,\ZZ)$-orbit, then $x_1\otimes 1, x_2\otimes 1$ are in the same grand $(f_\QQ,\QQ)$-orbit.

    By Lemma \ref{unipq}, for any $P\in M$, at most 2 distinct multiples of $P\otimes 1\in M\otimes_\ZZ\QQ$ can be in the same grand $(f_\QQ,\QQ)$-orbit, so at most 2 distinct multiples of $P$ can be in the same grand $(f,\ZZ)$-orbit, and the corollary follows from pigeonhole principle.
\end{proof}

\begin{theorem}[Theorem~\ref{main} (restated)]
Let $X/\Qbar$ be an abelian variety, and let $f:X\to X$ be an endomorphism of~$X$ (as an abstract variety) such that there is a point in $X(\Qbar)$ with $f$-orbit that is Zariski dense in $X$.
Then there is a number field~$K/\QQ$ such that~$X$ and~$f$ are defined over~$K$, and $(f,K)$ is weakly dense orbit transversal.
\end{theorem}
\begin{remark}
    To prove the theorem, we will construct a dense $(f,K)$-transversal for $X(K)$.
    In fact, it suffices to construct a set $S\subseteq X(K)$ such that $S$ is Zariski dense in $X$ and all points in $S$ are in distinct grand $(f,K)$-orbits.
    Then $S$ can be completed to an $(f,K)$-transversal.
    
    To do so, we will use Lemma \ref{densityreduction} (b) and Lemma \ref{indcoord} to choose a suitable point $P\in X(K)$ such that $\ZZ P=\{mP\mid m\in \ZZ\}$ is dense in $X$, which implies that $\overline{\ZZ P}$ is connected.
    Then by Lemma \ref{cyclicsubset}, any infinite subset of $\ZZ P$ will have the same Zariski closure as $\ZZ P$, and hence be dense in $X$. 
    This reduces the problem to finding an infinite subset $S\subseteq \ZZ P$ such that all points in $S$ are in distinct grand $(f,K)$-orbits, which will be resolved in a 2-case analysis.
\end{remark}
\begin{proof}
    By Lemma \ref{isoginv} and Lemma \ref{splitmap}, the assumption we are using and the property we are trying to prove are isogeny invariant, and we can always construct a lift of $f$ to a product of isotypic components isogenous to $X$.
    Therefore we may and do assume that $X$ is isomorphic to a product of its isotypic components.
    That is, $X\cong A_1\times \cdots \times A_n$, where $A_i$ is isogenous to $B_i^{k_i}$ for some simple abelian variety $B_i$, and $B_i$ is not isogenous to $B_j$ for $i\ne j$.
    Then $f$ decomposes as $(f_1,\dots ,f_n)$, where $f_i:A_i\to A_i$ and $f(x_1,\dots,x_n)=(f_1(x_1),\dots, f_n(x_n))$, and we may apply Lemma \ref{projorbit} to reduce the problem of finding points in distinct grand $(f,K)$-orbits in $X(K)$ to finding points in distinct grand $f_i$-orbits in $A_i(K)$ for an $i$.

    Following the argument in section~6 of \cite{ghiocascanlon}, we decompose each $A_i$ into $A_{i,1}+A_{i,2}$ as follows:
    Since $f$ is dominant, by the rigidity theorem, the map $f_i$ can be written as an isogeny $\t_i$ composed with a translation by $y_i\in A_i$, so $f_i(x)=\t_i(x)+y_i$.
    Up to replacing $f$ by $f^k$ for some appropriate $k$, which is allowed by Lemma \ref{iterinv}, we may assume that 
    \[\dim(\ker (\t_i^m-\id))=\dim(\ker (\t_i-\id))\]
    for all $i$.
    That is, the only eigenvalues of $\t_i$ that are roots of unity are now~1 for all $i$.
    Let $g_i\in \ZZ[t]$ be the minimal polynomial of $\t_i$, and let $r_i$ be the order of vanishing of~1 in $g_i$, then we may write $g_i(t)=g_i'(t)\cdot (t-1)^{r_i}$.
    Let $A_{i,1}=(\t_i-\id)^{r_i}(A_i)$ and $A_{i,2}=g_i'(\t)(A_i)$.
    
    By Lemma 6.1 and subsequent discussions in \cite{ghiocascanlon}, we know that $A_i=A_{i,1}+A_{i,2}$, where $A_{i,1}\cap A_{i,2}$ is finite, and up to conjugation by a translation we may assume that $f_i$ restricted to $A_{i,1}$ is an isogeny $\t_{i,1}:A_{i,1}\to A_{i,1}$ whose eigenvalues are not roots of unity, and $f_i$ restricted to $A_{i,2}$ is a unipotent isogeny $\t_{i,2}$ composed with a translation by $y_{i,2}\in A_{i,2}$, so $f_i|_{A_{i,2}}(x)=\t_{i,2}(x)+y_{i,2}$.
    We denote the restrictions as follows:
    \begin{align*}
        \s_{i,1}&:=f_i|_{A_{i,1}}:A_{i,1}\to A_{i,1},\ \s_{i,1}(x)=\t_{i,1}(x),\\
        \s_{i,2}&:=f_i|_{A_{i,2}}:A_{i,2}\to A_{i,2},\ \s_{i,2}(x)=\t_{i,2}(x)+y_{i,2}.
    \end{align*}
    
    Furthermore, for any $x\in A_i$, if we write it as $x_1+x_2$, possibly not uniquely, where $x_1\in A_{i,1}$ and $x_2=A_{i,2}$, then $f_i(x)=\s_{i,1}(x_1)+\s_{i,2}(x_2)$.
    This gives us the commutative diagram
\begin{center}
\begin{tikzcd}
{A_{i,1}\times A_{i,2}} \arrow[r, "f_i'"] \arrow[d, "\phi_i"] & {A_{i,1}\times A_{i,2}} \arrow[d, "\phi_i"] \\
A_i \arrow[r, "f_i"]                                       & A_i                                     
\end{tikzcd}
    \end{center}
    where $f_i'(x_1,x_2)=(\s_{i,1}(x_1),\s_{i,2}(x_2))$ and $\phi$ is an isogeny given by taking the quotient by a copy of $A_{i,1}\cap A_{i,2}$.
    Again by Lemma \ref{isoginv}, we may assume that $A_i\cong A_{i,1}\times A_{i,2}$, which reduces the problem of finding points in distinct grand $f_i$-orbits in $A_i$ to finding points in distinct grand $\s_{i,1}$-orbits in $A_{i,1}$ or distinct grand  $\s_{i,2}$-orbits in $A_{i,2}$ by Lemma \ref{projorbit}.

    Now we analyze the problem in~2 cases depending on the decomposition of the $A_i$'s.
    The first case is when there exists an $e$ such that $A_{e,1}$ is nontrivial.
    Recall that we are assuming
    $X\cong A_{1}\times \dots \times  A_{n}$, where $A_i$ are the isotypic components of $X$, and by the Poincar\'e reducibility theorem, the abelian variety $X$ is isogenous to 
    $B_1^{k_{1}}\times \dots \times B_n^{k_{n}}$, where $B_i$ is simple for all $i$, and $B_i$ is not isogenous to $B_j$ for $i\ne j$.
    
    Let $P\in X(K)$ be a point such that its image in 
    $B_i^{k_i}(K)$
    has $\End(B_i)$-linearly independent coordinates for all $i$, which is always possible after replacing $K$ with a finite extension.
    Then by Lemma \ref{indcoord}, we find that the set
    \[p_i(\ZZ P)=\ZZ \cdot p_i(P)=\{m\cdot p_i(P)\mid m\in\ZZ\}\]
    is Zariski dense in $A_i$ for all $i$.
    Since the projection of $\ZZ P$ is Zariski dense in $A_i$ for all $i$, we deduce that $\ZZ P$ is Zariski dense in $X$ by Lemma \ref{densityreduction}.

    If all multiples of $P$ are in distinct grand $(f,K)$-orbits, then we may complete $\ZZ P$ to an $(f,K)$-transversal for $X(K)$, which will be Zariski dense in $X$.
    Now suppose that there exists $c\ne d$ such that $cP$ and $dP$ are in the same grand $(f,K)$-orbit, then there exist integers $a,b\ge 0$ such that $cf^a(P)=df^b(P)$, and we claim that an iterate of $f$ restricts to multiplication by an integer on $A_{e,1}$.
    
    Since $X\cong (A_{1,1}\times A_{1,2})\times \cdots \times (A_{n,1}\times A_{n,2})$, we can write $P=(P_{1,1},P_{1,2},\ldots,P_{n,1},P_{n,2})$, where $P_{i,j}\in A_{i,j}(K)$.
    By construction, the point $P$ having image with $\End(B_i)$-linearly independent coordinates in 
    $B_i^{k_i}(K)$ for all $i$ implies that $\Ann(P_{e,1})=0$ in $\End(A_{e,1})$ by Lemma \ref{indcoord}.
    Then projecting to $A_{e,1}$ gives us the equality $c\s_{e,1}^a=d\s_{e,1}^b$ in $\End(A_{e,1})$, where $\s_{e,1}=\t_{e,1}$ is a self-isogeny on $A_{e,1}$ whose eigenvalues are not roots of unity.
    Without loss of generality let $a\le b$, then we can factor the equation into $\s_{e,1}^{a}(c-d\s_{e,1}^{b-a})=0$, which implies that $a<b$ and $c-d\s_{e,1}^{b-a}=0$, because neither $\s_{e,1}^a$ nor $(c-d)$ can equal 0.
    Then we find
    \[\s_{e,1}^{b-a}=\left[\frac{c}{d}\right]\] in $\End^0(A_{i,1})$.
    Since $\s_{e,1}$ satisfies a monic integer polynomial and the eigenvalues of the $\s_{e,1}$ are not roots of unity, the fraction $\frac{c}{d}$ must be an integer not equal to 1, and we write it as $\k=\frac{c}{d}$.
    Now replace $f$ by $f^{b-a}$, which is allowed by Lemma \ref{iterinv}, then $\s_{e,1}=[\k]$, and for any integer $m$, the grand $\s_{e,1}$-orbit of $mP_{e,1}$ consists of the points of the form $\k^jmP_{e,1}$ that are in $A_{e,1}(K)$.
    
    Let $S=\{mP\mid m\in \ZZ,\ \gcd(m, \k)=1\}$, then $S$ is an infinite subset of $\ZZ P$, so it is Zariski dense in $X$ by Lemma \ref{cyclicsubset}.
    Note that all points in $S$ are in distinct grand $(f,K)$-orbits, because the projection of $S$ to $A_{e,1}$, 
    \[\{mP_{e,1}\mid m\in \ZZ,\ \gcd(m, \k)=1\},\]
    consists of points that are in distinct grand $\s_{e,1}$-orbits by unique prime factorization of integers. 
    Therefore we may complete $S$ to an $(f,K)$-transversal for $X(K)$.
    This concludes the proof for the first case.

    In the second case, the component $A_{i,1}$ is trivial for all $i$, and $f_i=\s_{i,2}$ for all $i$.
    Therefore we have $f(x)=\t(x)+y$ on the entire $X$, where $\t=N+\id$ is a unipotent isogeny on $X$ with $N$ being nilpotent and $y\in X(K)$ is a non-torsion point.
    The fact that $y$ is non-torsion comes from the assumption that there exists a point $P_0$ with dense $f$-orbit.
    For any point $x\in X(K)$, we compute:
    \begin{align*}
        f^k(x)=&\f^k(x)+\sum_{i=0}^{k-1}\f^i(y)\\
        =&(\id+N)^k(x)+\sum_{i=0}^{k-1}(\id+N)^i(y)\\
        =&\sum_{j=0}^k{k \choose j}N^j(x)+\sum_{i=0}^{k-1}\left(\sum_{j=0}^i{i \choose j}N^j\right)(y)\\
        =&\sum_{j=0}^k{k \choose j}N^j(x)+\sum_{j=0}^{k-1}\left(\sum_{i=j}^{k-1}{i \choose j}N^j\right)(y)\\
        =&\sum_{j=0}^k{k \choose j}N^j(x)+\sum_{j=0}^{k-1}{k \choose j+1}N^j(y),
    \end{align*}
    which implies that the $f$-orbit of $x$ will lie in a union of translates of $N(X)$:
    \[\Orbit_f(x)\subseteq \bigcup_{i}\bigcup_m (N(X)+N^i(my)).\]
    Since $N$ is nilpotent, the index $i$ takes on finitely many values.
    If $y$ were torsion, then the index $m$ would take on finitely many values as well, and $\Orbit_f(x)$ would be contained in a finite union of translates of a proper abelian subvariety $N(X)$ for all $x$ in $X(K)$, contradicting the existence of a point with a dense orbit.
    
    Now let $P\in X(K)$ be a point such that its image in 
    $B_i^{k_i}(K)$
    has $\End(B_i)$-linearly independent coordinates for all $i$ and that $\{P,y\}$ is $\ZZ$-linearly independent, which is always possible after replacing $K$ with a finite extension.
    By Lemma \ref{indcoord}, the projection of $\ZZ P$ to $A_i$ is dense for all $i$, then by Lemma \ref{densityreduction}, the set $\ZZ P$ is Zariski dense in $X$.
    
    Consider $X(K)$ as a finitely generated $\ZZ$-module and $\t:X(K)\to X(K)$ as a unipotent $\ZZ$-linear map, then by Corollary \ref{unipz}, there exists an infinite subset $S\subseteq \ZZ P$ such that all points in $S$ are in distinct grand $(f,K)$-orbits.
    By Lemma \ref{cyclicsubset}, the set $S$ is infinite and therefore also Zariski dense in $X$, so we may complete $S$ to a Zariski dense $(f,K)$-transversal for $X(K)$, which finishes the proof.
\end{proof}

Lastly, we prove the strong orbit transversality for dominant maps on irreducible varieties over uncountable fields.

\begin{proposition}\label{uncountablefield}
Let $K$ be an uncountable field, $X/K$ an irreducible variety such that $X(K)$ is Zariski dense in $X$, and $f:X\to X$ a dominant morphism.
Then every $(f,K)$-transversal for $X(K)$ is Zariski dense in $X$.
\end{proposition}
\begin{proof}
    Let $S\subset X(K)$ be an $(f,K)$-transversal for $X(K)$ and assume for the sake of contradiction that $Z:=\overline{S}\subsetneq X$, i.e., the Zariski closure of $S$ is a proper closed subset of $X$, and hence $\dim(Z)<\dim(X)$.
    Since morphisms cannot increase the dimension of a variety, we get that $\dim(f^n(Z))\le \dim (Z)<\dim(X)$ and $\overline{f^n(Z)}$ is a proper closed subset of $X$ for all $n\ge 0$.
    On the other hand, because $f$ is dominant, the preimage of any proper closed subset is still a proper closed subset.
    Therefore $f^{-m}(\overline{f^n(Z)})$ is a proper closed subset of $X$ for all $m,n\ge 0$.
    Now observe that by definition of $(f,K)$-transversals, we have 
    \[X(K)=\bigcup_{m\ge 0}\bigcup_{n\ge0} f^{-m}(f^n(S)).\]
    Taking the Zariski closure on both sides, we get
    \[X=\overline{X(K)}=\overline{\bigcup_{m\ge 0}\bigcup_{n\ge0} f^{-m}(f^n(S))}\subseteq\bigcup_{m\ge 0}\bigcup_{n\ge0} f^{-m}(\overline{f^n(Z)}),\]
    because the rightmost term is closed and contains $\bigcup_{m\ge 0}\bigcup_{n\ge0} f^{-m}(f^n(S))$.
    This implies that $X$ is a countable union of proper subvarieties of $X$, which is a contradiction, so $\overline{S}=X$.
\end{proof}


\bibliographystyle{amsalpha} 
\bibliography{Citations}

\end{document}